\documentclass{article}

\usepackage{amsmath,amssymb,amsfonts,amsthm,latexsym,amstext,amscd}
\usepackage{epsfig,graphics,graphicx,color}
\usepackage{enumerate}
\usepackage{hyperref}

\newcommand{\bsx}{\boldsymbol{x}}%
\newcommand{\bsm}{\boldsymbol{m}}%
\newcommand{\bsj}{\boldsymbol{j}}%
\newcommand{\bsy}{\boldsymbol{y}}%
\newcommand{\bsk}{\boldsymbol{k}}%
\newcommand{\bszero}{\boldsymbol{0}}%

\newcommand{\NN}{\mathbb N}
\newcommand{\ZZ}{\mathbb Z}

\newcommand{\lcm}{{\rm lcm}}

\newtheorem{theorem}{Theorem}[section]
\newtheorem{lemma}[theorem]{Lemma}
\newtheorem{remark}[theorem]{Remark}

\theoremstyle{definition}

\newtheorem{definition}[theorem]{Definition}

\numberwithin{equation}{section}

\begin{document}

\title{A lower bound on the star discrepancy of generalized Halton sequences in rational bases}
\author       {Roswitha Hofer}


\maketitle
\begin{abstract}
In this paper we extend a result of Levin, who proved a lower bound on the star discrepancy of generalized Halton sequences in positive integer bases, to rational bases. 
\end{abstract}

\section{Introduction and main result}\label{sec:1}

For applications --- for instance in finance, physics, or digital imaging --- one relies on point distributions in the multidimensional unit cube that are uniformly spread. 
One important measure for the uniformity of a point set $\bsx_0,\bsx_1,\ldots,\bsx_{N-1}$ of $N$ points in $[0,1]^s$ is the star discrepancy $D^*_N$, defined by 
$$D_N^*(\bsx_0,\ldots,\bsx_{N-1})=\sup_{\bsy\in(0,1]^s}\left|\frac{\#\{0\leq n< N:x^{(i)}_n < y^{(i)}, \,i=1,\ldots,s\}}{N}-\prod_{i=1}^sy^{(i)}\right|,$$
where $x^{(i)}_n$ and $y^{(i)}$ denote the $i$th components of $\bsx_n$ and $\bsy$. 
For an infinite sequence $(\bsx_n)_{n\geq 0}$ in $[0,1]^s$, the star discrepancy $D^*_N$ is defined via the first $N$ elements of the sequence. The star discrepancy appears as one main magnitude in the celebrated Koksma--Hlawka inequality. This inequality gives an upper bound of the integration error of a quadrature rule that heavily depends on the star discrepancy of the sampling points. Hence, the smaller the discrepancy the better the approximation of the integral. Concerning this measure of uniformity the best explicit examples of sequences in dimension $s$ satisfy discrepancy bounds in the style of 
\begin{equation}\label{eq:ldseq}
ND_N^*\leq c \log^s N
\end{equation}
for $N$ large enough and with a positive constant $c$ that might depend on certain parameters but is independent of $N$. 
By a famous result of Schmidt \cite{schmidt} it is known that this bound is sharp in one dimension, as in this case a lower bound 
$$ND_N^*\geq c \log\, N$$
with an absolute constant $c>0$ holds for infinitely many $N$. 
In the $s$-dimensional case with $s>1$ the best known lower bound is due to Bilyk, Lacey and Vagharshakyan \cite{BLV}, and is of the form 
$$ND_N^*\geq c (\log\, N)^{s/2 +\eta}$$
for infinitely many $N$, with positive constants $c$ and $\eta$ that might depend on $s$ but are independent of $N$. This result improved and generalized earlier results by Roth \cite{roth}, Beck \cite{beck}, and Bilyk and Lacey \cite{BL}. 

In the theory of uniform distribution it is frequently conjectured, that \eqref{eq:ldseq} is best possible already. This is the reason why sequences satisfying a bound of the form \eqref{eq:ldseq} are called low-discrepancy sequences. 
In a series of four papers \cite{LevinH1,LevinH2,Levints1,Levints2} Levin is supporting this conjecture, by proving for many classes of low-discrepancy sequences a lower bound of the form 
$$\limsup_{N\to\infty}ND^*_N/\log^sN>0$$
where $s$ is the dimension of the sequence. 

In this paper we appeal to Levin's results of the papers \cite{LevinH1,LevinH2} for the ordinary and the generalized Halton sequences. For the definition of these sequences we introduce the $b$-adic radical inverse function $\varphi_b:\NN_0\to[0,1]$ for a fixed integer base $b\geq 2$. Let $n=\sum_{j=1}^\infty a_j(n)b^{j-1}$ be the unique base $b$ representation of $n$ with $a_j(n)\in\{0,1,\ldots,b-1\}$ then $$\varphi_b(n):=\sum_{j=1}^\infty a_j(n)/b^j.$$ This radical inverse function can be generalized by using a sequence of permutations $\Sigma=(\sigma_j)_{j\geq 1}$ on $\{0,1,\ldots,b-1\}$. Then we set 
$$\varphi^\Sigma_b(n):=\sum_{j=1}^\infty \sigma_j(a_j(n))/b^j.$$
Now for an $s$-dimensional ordinary Halton sequence $(H_{b_1,\ldots,b_s}(n))_{n\geq 0}$ choose $s$ pairwise coprime integer bases $b_1,\ldots,b_s\geq 2$ and set 
$$H_{b_1,\ldots,b_s}(n)=(\varphi_{b_1}(n),\ldots,\varphi_{b_s}(n)).$$
For an $s$-dimensional generalized Halton sequence $(H^\Sigma_{b_1,\ldots,b_s}(n))_{n\geq 0}$ choose $s$ pairwise coprime integer bases $b_1,\ldots,b_s\geq 2$ and in addition for each $i\in\{1,\dots,s\}$ a sequence of permutations $\Sigma_i=(\sigma_{i,j})_{j\geq 1}$ on $\{0,1,\ldots,b_i-1\}$. Finally, set 
$$H^\Sigma_{b_1,\ldots,b_s}(n)=(\varphi^{\Sigma_1}_{b_1}(n),\ldots,\varphi^{\Sigma_s}_{b_s}(n)).$$

In the starting paper \cite{LevinH1} Levin ensured 
$$\limsup_{N\to\infty}ND^*_N(H_{b_1,\ldots,b_s}(n))/\log^sN>0$$
by using an elementary method of proof. (This proof can be found with more details in the survey article \cite{KaltStock}). 

This method was generalized in \cite{LevinH2} and applied to several modified Halton sequences that are using Cantor's expansions, Neumann-Kakutani's $b$-adic adding machines, and digital permutations. One consequence of the result therein is 
$$\limsup_{N\to\infty}ND^*_N(H^\Sigma_{b_1,\ldots,b_s}(n))/\log^sN>0.$$
 
In this paper we consider a further modification of the Halton sequence, $$(H^\Sigma_{u_1/v_1,\ldots,u_s/v_s}(n))_{n\geq 0},$$ using rational bases $u_1/v_1,\ldots,u_s/v_s$ \cite{HoferRB} and prove for such sequences that 
$$\limsup_{N\to\infty}ND^*_N(H^\Sigma_{u_1/v_1,\ldots,u_s/v_s}(n))/\log^sN>0.$$
Our method of proof is following an appropriate modification of the methods introduced in \cite{LevinH1,LevinH2}. 

The paper is organized as follows. In Section~\ref{sec:2} we introduce an expansion of integers to rational bases, that is used to define a radical inverse function to rational bases. This is needed for the definition of the sequence $(H^\Sigma_{u_1/v_1,\ldots,u_s/v_s}(n))_{n\geq 0}$ and we state our main Theorem~\ref{thm:1}, which is proved in Section~\ref{sec:4}. Section~\ref{sec:3} collects some auxiliary lemmas that are needed for the proof of Theorem~\ref{thm:1}.

\section{An expansion to rational bases, generalized Halton sequences to rational bases, and the main Theorem~\ref{thm:1}}\label{sec:2}

Let $b\geq 2$ be an integer. It is well-known that every integer $z$ has a $b$-adic expansion of the form 
\begin{equation}\label{equ:leZ}
z=\sum_{r=1}^\infty a_rb^{r-1}\quad\text{ with } a_r\in\{0,1,\ldots,b-1\} \text{ for }r\in\NN. 
\end{equation}
In equation \eqref{equ:leZ} the coefficients $a_r$ can be computed by the following algorithm.
\begin{itemize}
\item[-] Set $z_0:=z$ and for each $r\geq 1$ set $z_{r}=(z_{r-1}-a_{r})/b$, where $a_{r}$ is the unique element in $\{0,1,\ldots,b-1\}$ such that $b|(z_{r-1}-a_{r})$. \\
\end{itemize}

First we generalize this algorithm to rational numbers $b=u/v$ with coprime integers $u\geq 2,\,v\geq 1$
\begin{itemize}
\item[-] Set $z_0:=z$ and for each $r\geq 1$ set $z_{r}=(v z_{r-1}-a_{r})/u$ where $a_{r}$ is the unique element in $\{0,1,\ldots,u-1\}$ such that $u|(vz_{r-1}-a_{r})$. \\
\end{itemize}
Note that $z_r$ is an integer for every $r\geq0$. 
Application of this algorithms produces the following \emph{formal} $u/v$-adic expansion of $z$
\begin{equation}\label{equ:FReZ}
\sum_{r=1}^\infty \frac{a_r}{v}\left(\frac{u}{v}\right)^{r-1}=:(a_1,a_2,\ldots)_{u/v}.
\end{equation} 
The expansion \eqref{equ:FReZ} coincides with the non-rational one in \eqref{equ:leZ} if $v=1$ and $u=b$. 
\begin{remark}{\rm
Note that we speak of a \emph{formal} expansion and do not take care about convergence of the series. 
Using induction we deduce for $j\in\NN_0$ the following identity 
\begin{equation}\label{id:1}
z=\sum_{r=1}^{j}\frac{a_r}{v}\left(\frac{u}{v}\right)^{r-1} + z_j\left(\frac{u}{v}\right)^{j}.
\end{equation}}

\end{remark}

\begin{remark}\label{rem:1}{\rm 
We call an expansion of the form \eqref{equ:FReZ} \emph{finite} if there are only finitely many nonzero $a_r$. 

If $u>v$ then \eqref{equ:FReZ} is finite for every nonnegative rational integer $z$ and it coincides with the rational base number system for $\NN_0$ considered by Akiyama et al. in \cite{AFS}. }
\end{remark}
 
\begin{remark}\label{rem:2} {\rm
Switching from the integer base $b$ to a rational base $u/v$ with coprime $u$ and $v$ is different from taking a digital permutation. Regard for example $u/v=3/2$. By Remark~\ref{rem:1} each $n\in\NN_0$ has a finite expansion. For the sake of simplicity we write $(a_1,a_2,\ldots,a_k)_{u/v}$ where $k$ is maximal such that $a_k\neq 0$ instead of $(a_1,a_2,\ldots)_{u/v}$. The number $0$ is denoted by $(0)_{u/v}$. In detail we obtain 
\begin{center}
\begin{tabular}{|l|l|l|l|l|}
\hline
$n$&$0$&$1$&$2$&$3$\\
\hline$(a_1,a_2,\ldots)_{3}$ &$(0)_{3}$&$(1)_{3}$&$(2)_{3}$&$(0,1)_{3}$\\
\hline$(a_1,a_2,\ldots)_{3/2}$ &$(0)_{3/2}$&$(2)_{3/2}$&$(1,2)_{3/2}$&$(0,1,2)_{3/2}$\\
\hline $n$&$4$&$5$&$6$&$7$\\
\hline$(a_1,a_2,\ldots)_{3}$&$(1,1)_{3}$&$(2,1)_{3}$&$(0,2)_{3}$&$(1,2)_{3}$\\
\hline$(a_1,a_2,\ldots)_{3/2}$&$(2,1,2)_{3/2}$&$(1,0,1,2)_{3/2}$&$(0,2,1,2)_{3/2}$&$(2,2,1,2)_{3/2}$\\
\hline $n$&$8$&$9$&10&11\\ 
\hline $(a_1,a_2,\ldots)_{3}$ &$(2,2)_{3}$&$(0,0,1)_{3}$&$(1,0,1)_{3}$&$(2,0,1)_{3}$\\
\hline $(a_1,a_2,\ldots)_{3/2}$ &$(1,1,0,1,2)_{3/2}$&$(0,0,2,1,2)_{3/2}$&$(2,0,2,1,2)_{3/2}$&$(1,2,2,1,2)_{3/2}$\\
\hline
\end{tabular}.
\end{center}
For instance, the behavior of the lengths of the finite expansions shows that switching from base $3$ to base $3/2$ cannot be described by digital permutations. 
}
\end{remark}

For our definition of Halton type sequences to rational bases we need first a proper radical inverse function. 

Let $u,v\in\NN$, satisfy $u\geq 2$ and $\gcd(u,v)=1$. Let $\Sigma=(\sigma_r)_{r\geq 1}$ be a sequence of permutations on $\{0,1,\ldots,u-1\}$. We define the \emph{$u/v$-adic radical inverse function} $\varphi^{\Sigma}_{u/v}:\NN_0\to[0,1]$ by 
$$n\mapsto \sum_{r\geq 1} \frac{\sigma_r(a_r)}{u^{r}}$$
where 
$$\sum_{r\geq 1} \frac{{a_r}}{v}\left(\frac{u}{v}\right)^{r-1}$$
is the formal $u/v$-adic expansion of $n$. 

We are finally in a position to define generalized Halton sequences to rational bases, which are low-discrepancy sequences (cf. \cite[Theorem~3]{HoferRB}). 
\begin{definition}\label{def:1}
Let $s$ be a dimension. Let $u_1,v_1,u_2,v_2,\ldots,u_s,v_s\in\NN$ satisfy $u_i\geq 2$, $\gcd(u_i,v_i)=1$ for $1\leq i\leq s$, and $\gcd(u_i,u_j)=1$ for all $1\leq i\neq j\leq s$. For each $i\in\{1,2,\ldots,s\}$ let $\Sigma_i=(\sigma_{i,r})_{r\geq 1}$ be a sequence of permutations on $\{0,1,\ldots,u_i-1\}$. Then the sequence $(\bsx_n)_{n\geq 0}$ in $[0,1]^s$ where the $n$th element $\bsx_n$ is given by 
$$\bsx_n=H_{u_1/v_1,\ldots,u_s/v_s}^{\Sigma}(n)=(\varphi^{\Sigma_1}_{u_1/v_1}(n),\ldots, \varphi^{\Sigma_s}_{u_s/v_s}(n))$$
is an \emph{$s$-dimensional generalized Halton sequence in bases $(u_1/v_1,u_2/v_2,\ldots,u_s/v_s)$}. 
\end{definition}
If $v_1=v_2=\cdots=v_s=1$ then the sequences in Definition~\ref{def:1} coincide with generalized Halton sequences, which were mentioned in Section~\ref{sec:1}. If, furthermore, $\sigma_{i,r}$ is the identity map on $\{0,1,\ldots,u_i-1\}$ for every $r\geq 1$ and $i=1, \ldots,s$ then Definition~\ref{def:1} gives the ordinary Halton sequence in pairwise coprime bases $(u_1,\ldots,u_s)$. Remark~\ref{rem:2} ensures that generalized Halton sequences in integer bases form a \emph{strict} subset of generalized Halton sequences in rational bases. 

Our main result in this paper is formulated in the following theorem, which generalizes results in \cite{LevinH1,LevinH2} and will be proved in Section~\ref{sec:4}. Section~\ref{sec:3} provides the main auxiliary results needed in Section~\ref{sec:4}. 
\begin{theorem}\label{thm:1}
Let $s\in\NN$. The star discrepancy of the $s$-dimensional generalized Halton sequence $(\bsx_n)_{n\geq 0}$ in Definition~\ref{def:1} satisfies 
$$\limsup_{N\to\infty}ND_N^*(\bsx_n)/\log^s N >0.$$
\end{theorem}

\section{Auxiliary Results}\label{sec:3}

\begin{lemma}\cite[Lemma~1]{HoferRB}\label{lem:1}
Let $u,v,j\in\NN$, satisfy $u\geq 2$ and $\gcd(u,v)=1$. Let $z^{(1)},z^{(2)}$ be integers and $(a_1^{(1)},a_2^{(1)},\ldots)_{u/v}$ be the $u/v$-adic expansion of $z^{(1)}$ and $(a_1^{(2)},a_2^{(2)},\ldots)_{u/v}$ the $u/v$-adic expansion of $z^{(2)}$. Then $a_r^{(1)}=a_r^{(2)}$ for every $1\leq r\leq j$ if and only if $z^{(1)}$ is congruent $z^{(2)} $ modulo $u^j$. 
\end{lemma}

For the formulation of the next lemma we introduce the truncation operator. 
By the construction of our sequence $(\bsx_n)_{n\geq 0}$ we have that each point of the $i$th component $(x_n^{(i)})_{n\geq 0}$, $i\in\{1,\ldots,s\}$, has prescribed $u_i$-adic digit expansion, where the case with almost all digits equal to $u_i-1$ is admissible. Now for $t\in\NN$ let $[x]_{t}$ denote $\sum_{j=1}^tx_ju_i^{-j}$ where we use the prescribed digit expansion of $x$ in base $u_i$ of the form $x=\sum_{j=1}^\infty x_ju_i^{-j}$ with $x_j\in\{0,1,\ldots, u_i-1\}$.

\begin{lemma}\label{lem:2}
Let $u_i,v_i$ and $\sigma_{i,r}$ for $i=1,\ldots,s$ and $r\geq 1$ be as in Definition~\ref{def:1}. Furthermore, let $k_i\in\NN$ and $y_i=\sum_{j=1}^{\infty}y_{i,j}u_i^{-j}$ with $y_{i,j}\in\{0,1,\ldots,u_i-1\}$ for $j\geq 1$ and $i=1,\ldots,s$. Furthermore, let $y_{i,k_i}>0$ for $i=1,\ldots,s$. We write $\bsk$ for $(k_1,\ldots,k_s)$ and $U_{\bsk}$ for $\prod_{i=1}^su_i^{k_i}$. 
Let $t\in\NN$ be such that $t\geq k_i$ for all $i=1,\ldots,s$. 
Let $\overline{v}_i\in\NN$ be such that $v_i\overline{v}_i\equiv 1 \pmod{u_i^{k_i}}$ and $M_{i,\bsk}$ be solving $$M_{i,\bsk}\equiv \left(\prod_{j=1,j\neq i}^su_j^{k_j}\right)^{-1}\pmod{u_i^{k_i}}.$$ 
Then 
\begin{enumerate}
\item for $i\in\{1,\ldots,s\}$ we have $[\varphi^{\Sigma_i}_{u_i/v_i}(n)]_{t}\in[[y_i]_{k_i},[y_i]_{k_i}+1/u_i^{k_i})$ if and only if
$$n\equiv \underbrace{\sum_{j=1}^{k_i}\sigma_{i,j}^{-1}(y_{i,j})u_i^{j-1}\overline{v}^j_i}_{=:\dot{y}_{i,k_i}}\pmod{{u_i^{k_i}}}.$$
\item $[\varphi^{\Sigma_i}_{u_i/v_i}(n)]_{t}\in[[y_i]_{k_i},[y_i]_{k_i}+1/u_i^{k_i})$ for all $i=1\ldots,s$ if and only if 
$$n\equiv \underbrace{{\sum_{i=1}^sM_{i,\bsk}U_{\bsk}\dot{y}_{i,k_i}u_i^{-k_i}}}_{=:\ddot{y}_{\bsk}}\pmod{U_{\bsk}}.$$

\item for $i\in\{1,\ldots,s\}$ we have $[\varphi^{\Sigma_i}_{u_i/v_i}(n)]_{t}\in[[y_i]_{k_i}-1/u_i^{k_i},[y_i]_{k_i})$ if and only if
$$n\equiv {\sum_{j=1}^{k_i-1}\sigma_{i,j}^{-1}(y_{i,j})u_i^{j-1}\overline{v}^j_i}+\sigma_{i,k_i}^{-1}(y_{i,j}-1)u_i^{k_i-1}\overline{v}^{k_i}_i\equiv \dot{y}_{i,k_i}+{b_i}u_i^{k_i-1}\overline{v}^{k_i}_i\pmod{{u_i^{k_i}}}$$
where $b_i\in\{1,\ldots,u_i-1\}$ is chosen such that 
$$b_i\equiv \sigma_{i,k_i}^{-1}(y_{i,j}-1)-\sigma_{i,k_i}^{-1}(y_{i,j})\pmod{u_i}.$$
\item $[\varphi^{\Sigma_i}_{u_i/v_i}(n)]_{t}\in[[y_i]_{k_i}-1/u_i^{k_i},[y_i]_{k_i})$ for all $i=1\ldots,s$ if and only if 
$$n\equiv \ddot{y}_{\bsk}+{\sum_{i=1}^sM_{i,\bsk}U_{\bsk}b_i\overline{v}^{k_i}_iu_i^{-1}}\pmod{U_{\bsk}}$$
with $b_i$ given in the last item. 
\item if $k_i'\geq k_i$ for all $i=1,\ldots,s$ then $\ddot{y}_{\bsk'}\equiv \ddot{y}_{\bsk}\pmod{U_{\bsk}}$. 
\item $v_i\overline{v}_i\equiv 1 \pmod{u_i^j}$ for all $j\in\{1,\ldots,k_i\}$. 
\item if $\alpha_i\in\NN$ solves $v_i^{\alpha_i}\equiv 1\pmod{u_i}$, then $\overline{v}_i^{\alpha_i}\equiv 1\pmod{u_i}$. 
\end{enumerate}
\end{lemma}
\begin{proof}
We only prove the first item, since the others are straightforward. 

Using \eqref{id:1} we have $[\varphi^{\Sigma_i}_{u_i/v_i}(n)]_{t}\in[[y_i]_{k_i},[y_i]_{k_i}+1/u_i^{k_i})$ if
$$n=\sum_{r=1}^{k_i}\frac{\sigma^{-1}_{i,r}(y_{i,r})}{v_i}\left(\frac{u_i}{v_i}\right)^{r-1}+z\left(\frac{u_i}{v_i}\right)^{k_i}$$
with certain $z\in\ZZ$. We multiply this equality with $v_i^{k_i}$ and obtain 
$$v_i^{k_i}n\equiv \sum_{r=1}^{k_i}\sigma^{-1}_{i,r}(y_{i,r}){u_i}^{r-1}v_i^{k_i-r} \pmod{u_i^{k_i}}.$$
Finally multiplication with $\overline{v_i}^{k_i}$ together with Lemma~\ref{lem:1} brings the desired result.  
\end{proof}

\begin{lemma}\label{lem:3}
Let $s>1$. Let $u_i,v_i$ be as in Definition~\ref{def:1}. Let $\overline{v}_i$ be such that $\overline{v}_iv_i\equiv 1 \pmod{u_i}$. Let $\tau_i\in\NN$ be such that $v_i^{\tau_i}\equiv 1 \pmod{u_i}$ and $b_i\in\{1,\ldots,u_i-1\}$ for $i=1,\ldots,s$. Then 
$$\sum_{i=1}^s\frac{b_i\overline{v}_i^{\tau_i}}{u_i}\not\equiv \frac{1}{2}\pmod{1}.$$
\end{lemma}
\begin{proof}
Let $d_i=\gcd(b_i,u_i)$, $\overline{b}_i=b_i/d_i$ and $\overline{u}_i=u_i/d_i$ and $\overline{u}_0=\overline{u}_1\cdots \overline{u}_s$. 
Then application of Lemma~\ref{lem:2} item 7 yields
$$\sum_{i=1}^s\frac{b_i\overline{v}_i^{\tau_i}}{u_i}\equiv \sum_{i=1}^s\frac{\overline{b}_i}{\overline{u}_i} \pmod{1}.$$
The rest of the proof follows the main ideas of \cite[Proof of Lemma~2]{LevinH2}. 
Assume first $\overline{u}_0\neq 0\pmod{2}$. Then the result immediately follows. 
If $\overline{u}_j\equiv 0\pmod{2}$ for some $j\in\{1,\ldots,s\}$ and
$$\sum_{i=1}^s\frac{\overline{b}_i}{\overline{u}_i} \equiv \frac{1}{2}\pmod{1}$$
then 
$$\underbrace{(\frac{\overline{u}_j}{2}-\overline{b}_j)/\overline{u}_j}_{=:c_1/\overline{u}_0}\equiv \underbrace{\sum_{i=1,i\neq j}^s\frac{\overline{b}_i}{\overline{u}_i}}_{=:c_2/\overline{u}_0}\pmod{1}$$ 
where $c_1$ and $c_2$ are integers such that $c_1\equiv c_2\pmod{\overline{u}_0}$. 
Now let $l\in\{1,\ldots,s\}$ such that $l\neq j$. Then 
$c_1\equiv 0 \pmod{\overline{u}_l}$ but $c_2\not\equiv 0 \pmod{\overline{u}_l}$, which yields the desired contradiction. 
\end{proof}

\section{Proof of Theorem~\ref{thm:1}}\label{sec:4}

If $s=1$ the result follows from the famous result of Schmidt \cite{schmidt}. So we can assume $s\geq 2$ in the following. 

For each $i\in\{1,\ldots,s\}$ we define $\tilde{u}_i:=u_1 \cdots u_s/u_i$, 
$$\beta_i:=\min\{1\leq k\leq \tilde{u}_i:u_i^k\equiv 1 \pmod{\tilde{u}_i}\},$$
and
$$\alpha_i:=\min\{1\leq k\leq u_i:v_i^k\equiv 1 \pmod{u_i}\}$$
and we set $\tau_i=\lcm(\alpha_i,\beta_i)$. 

We abbreviate $\max_{i\in\{1,\ldots,s\}}u_i$ with $u_0$ and define $T:=[(\log_{u_0}N)/s-1]$ for $N>u_0^{4su_0^{s+1}u_1\ldots u_s}$. 
For each $i\in\{1\ldots,s\}$ we define a sequence $(e_{i,r})_{r\geq 1}$ in $\{1,\ldots,u_i-1\}$ by setting $e_{i,r}\equiv \sigma_{i,r}^{-1}(0)-\sigma_{i,r}^{-1}(1)\pmod{u_i}$ and for each $b_i\in\{1,\ldots,u_i-1\}$ we regard the set 
$$\mathcal{L}^{(T)}_{i,b_i}:=\{1\leq r \leq T: \tau_i|r\mbox{ and }e_{i,r}=b_i\}$$ 
with its cardinality $L^{(T)}_{i,b_i}$. Let $b_i^{(T)}$ satisfy $$L^{(T)}_{i,b_i^{(T)}}=\max_{1\leq b_i\leq u_i}L^{(T)}_{i,b_i}.$$

Let $$m:=\min_{i=1,\ldots,s}L^{(T)}_{i,b_i^{(T)}}.$$ 
Note that with $\tau_0:=\max_{i\in\{1,\ldots,s\}}\tau_i\leq u_0^s$ we obtain the following chain of inequalities
\begin{equation}\label{eq:mLB}
m\geq T/(\tau_0 u_0)\geq T/u_0^{s+1}\geq \frac{\log N}{2s u_0^{s+1}\log u_0}>2u_1\ldots u_s
\end{equation}
since $N>u_0^{4su_0^{s+1}u_1\ldots u_s}$. 

Choose subsets 
$$\{k_{i,1}<\cdots<k_{i,m}\}\subseteq \mathcal{L}^{(T)}_{i,b^{(T)}_i}$$
and define 
$$y_i:=\sum_{j_i=1}^mu_i^{-k_{i,j_i}}$$
and $$[\bszero,\bsy):=\prod_{i=1}^s[0,y_i).$$
We abbreviate $(k_{1,j_1},\ldots,k_{s,j_s})$ to $\bsj$ and $(k_{1,m},\ldots,k_{s,m})$ to $\bsm$. 
Set $\overline{U}_{\bsj}:=\prod_{i=1}^su_i^{k_{i,j_i}}$. 
We observe that 
\begin{equation}\label{eq:UBUm}
2\overline{U}_{\bsm} = 2\prod_{i=1}^su_i^{k_{i,m}}\leq 2 u_0^{sT}\leq u_0^{s[s^{-1}\log_{u_0}N]}\leq N. 
\end{equation}
We define $t:=\max_{i\in\{1,\ldots,s\}}(\lceil \log_{u_i}N\rceil)$. Note that $T\leq t$. 
In the following we concentrate on the point set $([\bsx_n]_t)_{0\leq n <N}$ where 
$[\bsx_n]_t$ denotes $([x_n^{(1)}]_t,\ldots,[x_n^{(s)}]_t)$.

Let $w_m$ be the smallest nonnegative integer $<\overline{U}_{\bsm}$ such that 
$$[\bsx_{w_m}]_t\in\prod_{i=1}^s[y_i,y_i+u_i^{-k_{i,m}}).$$
Such a $w_m$ exists by the regularity of the point set $([\bsx_n]_t)_{0\leq n<N}$ which is ensured by Lemma~\ref{lem:1}. 

In the following we will concentrate on 
\begin{equation}\label{eq:CF}
\sum_{n=w_m}^{w_m+{M}-1}\left(\underline{1}_{[\bszero,\bsy)}([\bsx_n]_{t})-y_1\cdots y_s\right)
\end{equation}
with $1\leq M\leq \overline{U}_{\bsm}$. 

First we separate the interval $[\bszero,\bsy)$ into a union of disjoint intervals as follows
$$\bigcup_{1\leq j_1,\ldots,j_s\leq m}\underbrace{\prod_{i=1}^s\left[\sum_{l=1}^{j_i-1}u_i^{-k_{i,{l}}},\sum_{l=1}^{j_i}u_i^{-k_{i,l}}\right)}_{=:B(\bsj)}.$$
Note that $\lambda_s(B(\bsj))=1/\overline{U}_{\bsj}$. 

Then \eqref{eq:CF} equals 
$$\sum_{1\leq j_1,\ldots ,j_s\leq m}\underbrace{\sum_{n=w_m}^{w_m+{M-1}}\left(\underline{1}_{B(\bsj)}([\bsx_n]_{t})-\frac{1}{\overline{U}_{\bsj}}\right)}_{\Sigma_{\bsj,M}}=:\Sigma_{M}.$$

The core of the proof is to compute the average 
$$\alpha_m:=\frac{1}{\overline{U}_{\bsm}}\sum_{M=1}^{\overline{U}_{\bsm}}\Sigma_{M}.$$
We write $M=M_1\overline{U}_{\bsj}+M_2$ with $0< M_2\leq \overline{U}_{\bsj}$ and using the regularity of the point set $([\bsx_n]_t)_{0\leq n<N}$ we observe 
\begin{equation*}
\Sigma_{\bsj,M_1\overline{U}_{\bsj}}=0. 
\end{equation*}
Hence 
\begin{equation*}
\Sigma_{\bsj,M}=\sum_{n=w_m+M_1\overline{U}_{\bsj}}^{w_m+M_1\overline{U}_{\bsj}+{M_2-1}}\left(\underline{1}_{B(\bsj)}([\bsx_n]_{t})-\frac{1}{\overline{U}_{\bsj}}\right).
\end{equation*}
In the following we define $\overline{v}_{i,m}\in\NN$ such that $\overline{v}_{i,m}v_i\equiv 1 \pmod{u_i^{k_{i,m}}}$ and observe with the usage of Lemma~\ref{lem:2} items 4, 5, and 6 that
\begin{align*}
\underline{1}_{B(\bsj)}([\bsx_n]_{t}) \Leftrightarrow n\equiv w_m+\sum_{i=1}^s\overline{M}_{i,\bsj}\overline{U}_{\bsj}b_i^{(T)}\overline{v}^{k_{i,j_i}}_{i,m}u_i^{-1} \pmod{\overline{U}_{\bsj}}
\end{align*}
where $\overline{M}_{i,\bsj}=M_{i,(k_{1,j_1},\ldots,k_{s,j_s})}$. We define $A_{\bsj}\in\{0,1,\ldots,\overline{U}_{\bsj}-1\}$ such that 
\begin{equation}\label{eq:Aj}
A_{\bsj}\equiv \sum_{i=1}^s\overline{M}_{i,\bsj}\overline{U}_{\bsj}b_i^{(T)}\overline{v}^{k_{i,j_i}}_{i,m}u_i^{-1} \pmod{\overline{U}_{\bsj}}
\end{equation}
then 
$$\Sigma_{\bsj,M}=\underline{1}_{[0,M_2)}(A_{\bsj})-M_2\overline{U}^{-1}_{\bsj}.$$

Altogether it is not so hard to see that 
\begin{equation}\label{eq:alphamvalue}
\alpha_m=\sum_{1\leq j_1,\ldots,j_s\leq m} \frac{1}{\overline{U}_{\bsm}}\sum_{M=1}^{\overline{U}_{\bsm}}\Sigma_{\bsj,M}=\sum_{1\leq j_1,\ldots,j_s\leq m}\left(\frac12 -\frac{A_{\bsj}}{\overline{U}_{\bsj}}-\frac{1}{2\overline{U}_{\bsj}}\right)
\end{equation}
as 
\begin{align*}
\frac{1}{\overline{U}_{\bsm}}\sum_{M=1}^{\overline{U}_{\bsm}}\Sigma_{\bsj,M}&=\frac{1}{\overline{U}_{\bsj}}\sum_{M_2=1}^{\overline{U}_{\bsj}}\left(\underline{1}_{[0,M_2)}(A_{\bsj})-M_2\overline{U}^{-1}_{\bsj}\right)\\
&=1-\frac{A_{\bsj}}{\overline{U}_{\bsj}}-\frac{\overline{U}_{\bsj}+1}{2 \overline{U}_{\bsj}}.
\end{align*}
For the next step we observe, that by the definition of the $k_{i,{j_i}}$ we know that $\alpha_i|k_{i,{j_i}}$ as well as $\beta_i|k_{i,{j_i}}$. 
The first together with Lemma~\ref{lem:2} item 7 ensures $\overline{v}_{i,m}^{k_{i,j_i}}\equiv 1\pmod{u_i}$. 
The second ensures that for $l\neq i$, $u_l^{k_{l,j_l}}\equiv 1 \pmod{u_i}$. Hence $\overline{M}_{i,\bsj}\equiv 1 \pmod{u_i}$.

We apply Lemma \ref{lem:3} together with \eqref{eq:Aj} and obtain
$$\frac{A_{\bsj}}{\overline{U}_{\bsj}}\equiv \frac{c}{u_1\cdots u_s}\not\equiv \frac{1}{2}\pmod{1}$$
with $c\in\{0,1,\ldots,u_1\cdots u_s-1\}$ independent of $\bsj$. 
This, \eqref{eq:alphamvalue}, and the fact that $\overline{U}_{\bsj}\geq 2^{j_1+\cdots+j_s}$ yields
$$|\alpha_m|\geq \frac{m^s}{4u_1\cdots u_s} \mbox{ as }m\geq 2u_1\cdots u_s.$$
The latter is guaranteed by \eqref{eq:mLB}.

Hence with $C:=1/(2^{s+2}s^su_1\cdots u_s u_0^{s^2+s}\log^s u_0)$ and bearing in mind \eqref{eq:UBUm} we obtain the following chain of inequalities
\begin{align*}
C\log^s N\leq \frac{m^s}{4u_1\cdots u_s} \leq |\alpha_m|\leq \sup_{1\leq M\leq \overline{U}_{\bsm}}MD_M^*([\bsx_{n+w_m}]_t)\\
\leq \sup_{1\leq L,L+M\leq 2\overline{U}_{\bsm}}MD_M^*([\bsx_{n+L}]_t)\leq 2 \sup_{1\leq M\leq N}MD_M^*([\bsx_n]_t).
\end{align*}

This lower bound is valid for the truncated point set $([\bsx_n]_t)_{0\leq n<N}$. We can get rid of the truncation by performing an $\varepsilon$-shift following an idea of Niederreiter and \"{O}zbudak \cite[Proof of Lemma~4.2]{NieOez} and using a general principle of \cite[Proof of Lemma~2.5]{niesiam} (see also \cite[End of the proof of Theorem~3]{HoferRB}).

Altogether, Theorem~\ref{thm:1} follows. 

\section*{Acknowledgments}

The author is supported by the Austrian Science Fund (FWF): Project F5505-N26, which is a part of the Special Research Program ``Quasi-Monte Carlo Methods: Theory and Applications''. Furthermore, the author appreciates the valuable comments of Isabel Pirsic.

Roswitha Hofer, Institute of Financial Mathematics and Applied Number Theory, Johannes Kepler University Linz, Altenbergerstr. 69, 4040 Linz, AUSTRIA, roswitha.hofer@jku.at

\end{document}